\documentclass[a4paper,11pt,reqno]{amsart}
\usepackage{titlecaps}
\usepackage{amsfonts}
\usepackage{amsmath}
\usepackage{amssymb}
\usepackage{cite}
\usepackage{mathrsfs}
\usepackage{enumitem}
\usepackage{xcolor}
\usepackage{graphicx}
\usepackage[pagebackref,hypertexnames=false, colorlinks, citecolor=red,linkcolor=blue, urlcolor=red]{hyperref}

\newtheorem{theorem}{Theorem}[section]
\newtheorem{lemma}[theorem]{Lemma}

\newtheorem{proposition}[theorem]{Proposition}
\theoremstyle{definition}
\newtheorem{definition}[theorem]{Definition}
\newtheorem{example}[theorem]{Example}

\newtheorem{remark}[theorem]{Remark}
\theoremstyle{remark}

\numberwithin{equation}{section}

\begin{document}
	
	\title{On integral representation of radial operators}
 
	\author[Bishal Bhunia]{Bishal Bhunia $^\dag$}
	\address{
		Bishal Bhunia:
		\endgraf
		Department of Mathematics and Statistics
		\endgraf
		Indian Institute of Science Education and Research Kolkata
		\endgraf
		Mohanpur-741246 
		\endgraf
		India
	}
	\email{bb22rs011@iiserkol.ac.in}

 \thanks{$^\dag$ This work is supported by University Grant Commission Fellowship (NTA reference no. 211610189115) availed at Indian Institute of Science Education and Research Kolkata}
	\date{}
	
	\begin{abstract}
  In this article, we characterize the radial operators on weighted Bergman spaces of Reinhardt domains in $\mathbb{C}^n$, the Dirichlet and the Hardy spaces of the open unit disk $\mathbb{D}$, in terms of integral representations. We also investigate normality, compactness, spectrum and numerical ranges of these operators. Further, utilizing the theory of radial operators, we produce examples of von Neumann algebras of analytic functions on any Reinhardt domain.
	\end{abstract}
	\maketitle
	
	\section{Introduction}
The notion of radial operator on the Bergman space of the unit disc in $\mathbb{C}$ was first introduced by N. Zorboska in \cite{zorboska2003berezin}. Subsequently, this was investigated in several domains, such as unit polydisc $\mathbb D^n$, unit ball $\mathbb B^n$ in $\mathbb C^n$ \cite{bauer2014eigenvalue, li2019radial}. In this article, we study them in the general framework of the Reinhardt domain. Let $H$ be a Hilbert space of holomorphic functions on a Reinhardt domain $\Omega$ in $\mathbb C^n$, that is, the entrywise product 
$$\lambda z := (\lambda_1 z_1, \cdots, \lambda_n z_n)\in \Omega \mbox{~for~} z\in\Omega \mbox{~and~}\lambda = (\lambda_1, \cdots, \lambda_n) \in \mathbb{T}^n,$$ 
where $\mathbb{T}^n:=\{(\lambda_1, \cdots, \lambda_n): |\lambda_j|=1 , \forall 1\leq j \leq n\}$. We consider an operator $V_\lambda: H \rightarrow H$ defined by
$$
\left(V_\lambda f\right)(z)=f(\lambda z), \forall f \in H , z\in\Omega.
$$
A bounded linear operator $R$ on $H$ is said to be \textit{radial} if each $V_\lambda$ is bounded on $H$ and $R$ commutes with $V_\lambda$ for all $\lambda \in \mathbb{T}^n$, that is, $RV_\lambda=V_\lambda R$ for all $\lambda \in \mathbb{T}^n$.
Let $\phi$ be a holomorphic function from $\Omega$ to itself. The composition operator on $H$ induced by $\phi$, is denoted by $C_\phi$, and is defined as $C_\phi(f) = f \circ \phi$, for all $f \in H$. Then a radial operator on $H$ is can be viewed as a commutant of all composition operators of the form $C_\phi$, where $\phi$ is a linear map on $\mathbb{C}^n$ such that the the matrix representation of $\phi$ with respect to standard basis of $\mathbb{C}^n$ is unitary and diagonal.

Let $\omega$ be an admissible, multi-radial weight on $\Omega$ \cite[page - 5]{chakrabarti2024projections} and $d V$ the Lebesgue (volume) measure on $\Omega$. Let $A^2(\Omega, \omega)$ be the weighted Bergman space (with respect to the weight $\omega$) on $\Omega$, i.e., the space of holomorphic functions in $L^2(\Omega, \omega d V)$. Let $F^2(\mathbb{C})$, $H^2(\mathbb{D})$ and $\mathscr{D}(\mathbb{D})$ denote the Fock space of $\mathbb{C}$, the Hardy and the Dirichlet spaces of the open unit disc $\mathbb{D}$ respectively.

In \cite{zhu2015singular}, Zhu posed the following problem on the Fock space $F^2(\mathbb{C})$ of characterizing all $\phi \in F^2(\mathbb{C})$ for which the integral operator
$$
S_\phi f(z)=\int_{\mathbb{C}} f(w) e^{z \bar{w}} \phi(z-\bar{w}) d A(w)
$$
is bounded on $F^2(\mathbb{C}),$ where $dA$ denotes the Lebesgue (area) measure on $\mathbb{C}$.
In \cite{cao2020boundedness}, Cao et al. solved this question for the Fock space $F^2\left(\mathbb{C}^n\right)$. Motivated by this, Pinlodi and Naidu in \cite{mohan2024integral} solved a similar problem for the (unweighted) Bergman space $A^2(\mathbb{D})$ as follows: Characterize all holomorphic functions $g$ on $\mathbb{D}$ so that the integral operator

\begin{equation*}
    \left(S_g f\right)(z)=\int_{\mathbb{D}} f(w) g(z \bar{w}) d A(w)
\end{equation*} 
is bounded on $A^2(\mathbb{D})$. It turns out that every radial operator on $A^2(\mathbb{D})$ is of the form $S_g$ for some holomorphic functions $g$ on $\mathbb{D}$. Inspired by these advancements, we prove that the bounded radial operators $R$ on $A^2(\Omega, \omega)$, $H^2(\mathbb{D})$ and $\mathscr{D}(\mathbb{D})$ will have the following integral representations, respectively:

$${Rf}(z)=\int_{\Omega} f(w) g(z \bar{w}) \omega d V(w),\mbox{~on~} A^2(\Omega, \omega),$$  

$${R} f(z)=\hat{u}(0) f(0)+z \int_{\mathbb{D}} f^{\prime}(w) g^{\prime}(z \bar{w}) \log \frac{1}{|z|^2} d A(w), \mbox{~on~} H^2(\mathbb{D}),$$

$$R f(z)=\hat{u}(0) f(0)+z \int_{\mathbb{D}} f'(w) g'(z \bar{w}) d A(w), \mbox{~on~} \mathscr{D}(\mathbb{D}), $$
where $\hat{u}(0)$ denotes the $0-$th Fourier coefficient of the periodic distribution $u$, which induces the holomorphic function $g$.
In a recent paper \cite{Zhu2024}, Ma and Zhu gave the first example of $C^*$-algebra consisting of entire functions on $\mathbb{C},$ and posed the challenge of coming up with more examples of $C^*$-algebra of analytic functions on other complex domains. An example of a von Neumann algebra of analytic functions on $\mathbb{B}^n$ was provided by Pinlodi and Naidu in \cite{venkuvonneumann}. By applying the ideas presented in section \ref{SEC:p=3}, we produce a large class of examples of von Neumann algebras on all Reinhardt domains (see section \ref{SEC: p=4}), answering the challenge posed by Ma and Zhu in \cite{Zhu2024} in a more general setting.  We also state several operator theoretic properties of the radial operators such as normality, compactness, spectrum, numerical range and characterization of common reducing subspaces. Lastly, in section \ref{SEC: p=5}, using the integral representation of radial operators on the Dirichlet space of $\mathbb{D}$ and the Hardy space of $\mathbb{D}$, we provide the corresponding examples of von Neumann algebras of analytic function on $\mathbb{D}$ and compare them. In the next section, we present all basic definitions and concepts that are required to substantiate our article.

	
\section{Preliminaries}\label{SEC:preliminaries}
	
	Let $H$ be a separable Hilbert space. $B(H)$ will denote the set of all bounded linear operators on $H$. Let $\mathbb{Z}^n$ be the collection of all $n$-tuples of integers and $\Lambda$ a subset of $\mathbb{Z}^n$. Let $l^2:=l^2(\Lambda)$ be the space of all sequences $(a_\alpha)_{\alpha\in \Lambda}$ such that
$$
\sum_{\alpha\in \Lambda}\left|a_\alpha\right|^2<\infty
$$
and $l^{\infty}:=l^{\infty}(\Lambda)$ be the space of all bounded sequences.\\\\

\begin{definition}(Periodic functions)\cite[Definition 3.1.4]{ruzhansky2009pseudo} 
A function $f: \mathbb{R}^n \rightarrow \mathbb{C}$ is said to be 1-periodic if $f(x+k)=f(x)$ for every $x \in \mathbb{R}^n$ and $k \in \mathbb{Z}^n$.
\end{definition}
Let $\mathbb{T}^n:=\mathbb{R}^n / \mathbb{Z}^n=\{x+\mathbb{Z}^n: x \in \mathbb{R}^n\}$. Now onwards, we shall consider all 1-periodic functions on $\mathbb{R}^n$ as functions defined on $\mathbb{T}^n$. Let $C^m(\mathbb{T}^n)$ denote the space of all 1-periodic $m$ times continuously differentiable functions on $\mathbb{T}^n$ and the set of all test functions is given by
$$
C^{\infty}(\mathbb{T}^n):=\bigcap_{m \in \mathbb{Z}_{+}} C^m(\mathbb{T}^n)
$$


\begin{definition} \cite[Definition 3.1.6]{ruzhansky2009pseudo}
The space of \textit{rapidly decaying sequences} indexed by $\Lambda \subseteq \mathbb{Z}^n$ is denoted by $\mathcal{S}(\Lambda)$. A complex-valued sequences $a=\left(a_\alpha\right)_{\alpha \in \Lambda} \in \mathcal{S}(\Lambda)$ if for any $M < \infty$, there exists a constant $C_{a,M}$ such that 
$$
|a_{\alpha}| \leq C_{a,M} (1 + \|\alpha\|)^{-{M/2}},
$$
holds for all $\alpha \in \mathbb{Z}^n$, where $\|.\|$ denotes the usual Euclidean norm.
\end{definition}
\begin{definition}\cite[p. 300]{ruzhansky2009pseudo}
A complex sequence $\left(u_\alpha\right)_{\alpha \in \Lambda}$ is said to be of \textit{slow growth} if there exist constants $N < \infty$ and $C_{u,N}$ such that
$$
\left|a_\alpha\right| \leq C_{a,N}(1+\|\alpha\|)^{N/2}, \forall \alpha \in \Lambda .
$$

The set of all sequences of slow growth is denoted by $\mathcal{S}'(\Lambda).$

\end{definition}

\begin{definition} (Periodic distribution space $\mathcal{D}^{\prime}(\mathbb{T}^n)$ ). \cite[Definition 3.1.25]{ruzhansky2009pseudo} 
A continuous linear functional $T: C^{\infty}(\mathbb{T}^n) \rightarrow \mathbb{C}$ is called periodic distribution. The set of all periodic distributions is denoted by $\mathcal{D}^{\prime}(\mathbb{T}^n)$. For $u \in \mathcal{D}^{\prime}(\mathbb{T}^n)$ and $\phi \in C^{\infty}(\mathbb{T}^n)$, we denote the action of $u$ on $\phi$ by
$$
u(\phi)=\langle u, \phi\rangle .
$$

For any $\psi \in C^{\infty}(\mathbb{T}^n)$,
$$
\phi \mapsto \int_{\mathbb{T}^n} \phi(x) \psi(x) d x
$$
is a periodic distribution, which gives the embedding $\psi \in C^{\infty}(\mathbb{T}^n) \subset \mathcal{D}^{\prime}(\mathbb{T}^n)$.
The topology of $\mathcal{D}^{\prime}(\mathbb{T}^n)$ is the weak*-topology. We refer the reader to \cite{ruzhansky2009pseudo} for further details.
\end{definition}

\begin{definition}(Space $L^2(\mathbb{T}^n)$ ). \cite[Definition 3.1.11]{ruzhansky2009pseudo} The space $L^2(\mathbb{T}^n)$ is a Hilbert space with the inner product
$$
\langle u, v\rangle_{L^2(\mathbb{T}^n)}:=\int_{\mathbb{T}^n} u(x) \overline{v(x)} d x .
$$
\end{definition}

\begin{definition}(Periodic Fourier transform). \cite[Definition 3.1.8]{ruzhansky2009pseudo} Let
$$
F=(f \mapsto \widehat{f}): C^{\infty}(\mathbb{T}^n) \rightarrow \mathcal{S}(\mathbb{Z}^n)
$$
be the Fourier transform defined by
$$
\widehat{f}(\alpha):=\int_{\mathbb{T}^n} e^{-i \alpha \cdot \theta} f(\theta) d \theta, \quad \alpha \in \mathbb{Z}^n.
$$

Then $F$ is a bijection and its inverse $F^{-1}=(a \mapsto \check{a}): \mathcal{S}(\mathbb{Z}^n) \rightarrow C^{\infty}(\mathbb{T}^n)$ is given by
$$
\check{a}(x):=\sum_{\alpha \in \mathbb{Z}^n} e^{i \alpha \cdot \theta} a(\alpha) .
$$

It is well-known that the Fourier transform $F: L^2(\mathbb{T}^n) \rightarrow l^2(\mathbb{Z}^n)$ is a unitary operator. Let
$L_\Lambda^2\left(\mathbb{T}^n\right)$ denote the closed subspace of $L^2(\mathbb{T}^n)$ consisting of functions having Fourier coeffcients zero on $\mathbb{Z}^n  \setminus \Lambda.$ The restriction of $F$ on any of the closed subspace of $L^2(\mathbb{T}^n)$ will also be denoted by $F$.
\end{definition}
\begin{definition} (Fourier transform on $\mathcal{D}^{\prime}(\mathbb{T}^n)$ ). \cite[Definition 3.1.27]{ruzhansky2009pseudo} The Fourier transform is extended uniquely to the mapping $F: \mathcal{D}^{\prime}(\mathbb{T}^n) \rightarrow \mathcal{S}^{\prime}(\mathbb{Z}^n)$ by the formula
$$
\langle F u, \phi\rangle:=\left\langle u, \iota \circ F^{-1} \phi\right\rangle,
$$
where $u \in \mathcal{D}^{\prime}(\mathbb{T}^n), \phi \in \mathcal{S}(\mathbb{Z}^n)$, and $\iota$ is defined by $(\iota \circ \psi)(x)=\psi(-x)$.
\end{definition}
\begin{definition} \cite[Definition 3.186]{iorio2001fourier} Let $f, g \in \mathcal{D}^{\prime}(\mathbb{T}^n)$. The convolution $f * g$ is defined by the formula $\langle f * g, \phi\rangle=\langle f,(\iota \circ g) * \phi\rangle, \quad \phi \in C^{\infty}(\mathbb{T}^n)$, where $\iota \circ g$ is as defined in Definition 2.7.
\end{definition}
\begin{theorem} \cite[Corollary 3.188]{iorio2001fourier} Let $f, g \in \mathcal{D}^{\prime}(\mathbb{T}^n)$. Then $f * g \in \mathcal{D}^{\prime}(\mathbb{T}^n)$ and $\widehat{(f * g)}(\alpha)=\widehat{f}(\alpha) \widehat{g}(\alpha)$ for all $\alpha \in \mathbb{Z}^n$.
\end{theorem}

Now, let us recall some well-known definitions.

\begin{definition}
Let $T \in B(H)$. The spectrum of $T$, denoted by $\sigma(T)$, consists of all $z \in \mathbb{C}$ such that $(T-z I)^{-1} \notin B(H)$.
A complex number $\lambda$ is called an eigenvalue of $T$ if there exists $v \in H \backslash\{0\}$ satisfying $T v=\lambda v$. The set of all eigenvalues of $T$ is said to be the point spectrum of $T$ and is denoted by $\sigma_p(T)$. The number $\lambda$ is said to be an approximate eigenvalue of $T$ if there exists a sequence $\left(v_n\right)$ of unit vectors in $H$ such that $(T-\lambda I) v_n \rightarrow 0$ as $n \rightarrow \infty$. The collection of approximate eigenvalues of $T$ is called the approximate point spectrum of $T$ and it is denoted by $\sigma_{app}(T)$. Clearly, $\sigma_p(T) \subseteq \sigma_{app}(T) \subseteq \sigma(T)$. 
Let $a \in l^{\infty}$. Define the multiplication operator $M_a: l^2 \rightarrow l^2$ by
$$
\left(M_a b\right)(m)=a(m) b(m) \text { for all } b \in l^2 \text { and } m \in \Lambda .
$$

It is known that $M_a$ is bounded on $l^2$ if and only if $a \in l^{\infty}$. Moreover, $\left\|M_a\right\|_{l^2 \rightarrow l^2}=\|a\|_{l^{\infty}}$. If $\mathcal{M}\left(l^2\right)=$ $\left\{M_a: a \in l^{\infty}\right\}$, then the map $I: l^{\infty} \rightarrow \mathcal{M}\left(l^2\right)$ defined by $I(a)=M_a$ is a $*$-isometric isomorphism.

\end{definition}

\begin{theorem}\cite{bhatia2009notes} \label{multiplication}
\begin{enumerate}
\item $M_a^*=M_{\bar{a}}$, where $\bar{a}(m)=\overline{a(m)}, \quad m \in \Lambda$.

\item $M_{a_1} M_{a_2}=M_{a_1 a_2}=M_{a_2 a_1}=M_{a_2} M_{a_1}$.



\item The collection $\mathcal{M}\left(l^2\right)$ is a maximal commutative $C^*$-subalgebra of $B\left(l^2\right)$.

\item $\sigma\left(M_a\right)=\sigma_{app}\left(M_a\right)=\overline{\operatorname{range}(a)}$.

\item $\lambda \in \sigma_p\left(M_a\right)$ if and only if $a(m)=\lambda$ for some $m \in \Lambda$.

\item $M_a$ is compact if and only if $(a(m))_{m \in \Lambda}$ vanishes at infinity.

\item $M_a$ is of finite rank if and only if $a(m)=0$ for all but finitely many $m \in \Lambda$.
\end{enumerate}
\end{theorem}
We recall the following theorem.

\begin{theorem} \cite[Theorem 3.6.3]{grafakos2008classical}\label{thm 2.12}
A bounded linear operator $T$ commutes with all translations on $L^2(\mathbb{T}^n)$ if and only if there exists a bounded sequence $a=\left(a_\alpha\right)_{\alpha \in \mathbb{Z}^n}$ such that
$$
(T f)(x)=\sum_{\alpha \in \mathbb{Z}^n} a_\alpha \widehat{f}(\alpha) e^{i \alpha \cdot \theta}
$$
for all $f \in C^{\infty}(\mathbb{T}^n)$. Moreover, in this case we have $\|T\|_{L^2(\mathbb{T}^n) \longrightarrow L^2(\mathbb{T}^n)}=\left\|(a_\alpha)\right\|_{\infty}$.
\end{theorem}

\begin{theorem} \cite[Theorem 3.172]{iorio2001fourier} Let $a=(a_\alpha)_{\alpha \in \mathbb{Z}^n} \in \mathcal{S}^{\prime}(\mathbb{Z}^n)$. Then there exists a unique $u \in \mathcal{D}^{\prime}(\mathbb{T}^n)$ such that $\hat{u}=a$. Conversely, if $u \in \mathcal{D}^{\prime}(\mathbb{T}^n)$ then $\hat{u} \in \mathcal{S}^{\prime}(\mathbb{Z}^n)$. In other words, $\widehat{\mathcal{D}^{\prime}(\mathbb{T}^n)}=\mathcal{S}^{\prime}(\mathbb{Z}^n)$.
\end{theorem}

\begin{definition}
    
Let $y \in \mathbb{R}^n$ be fixed. Then for a measurable function $f$ on $\mathbb{T}^n$, we define translation of $f$ by
$$
\left(\tau_y f\right)(x)=f(x-y), \forall x \in \mathbb{R}^n .
$$
\end{definition}

	\section{Radial operators on Weighted Bergman spaces of Reinhardt domains}\label{SEC:p=3}	

 An open subset $\Omega$ of $\mathbb{C}^n$ is said to be a \textit{Reinhardt domain} if $z \in \Omega$ and $\lambda \in \mathbb{T}^n$ implies $\lambda z \in \Omega$. Some examples of such domains are the unit ball, the unit polydisc, poly-annulus and complements of those in $\mathbb{C}^n$. This class of domains is covered in great detail in the monograph \cite{jarnicki2008first}. Throughout this section and the next, $\Omega$ will be assumed to be bounded and connected Reinhardt domain in $\mathbb{C}^n$.

The \textit{Reinhardt shadow} of $\Omega$, denoted by $|\Omega|$, is the set
 $$ |\Omega| := \{(|z_1|, |z_2|, \hdots, |z_n|) : (z_1, z_2, \hdots, z_n) \in \Omega \} \subseteq (\mathbb{R}^+ \cup \{0\})^n.$$

 A function $\omega: \Omega \rightarrow [0, \infty]$ which is measurable and positive almost everywhere, is said to be a \textit{multi-radial weight} on $\Omega$ if there exist a function $a: |\Omega| \rightarrow[0, \infty]$ such that

 $$
\omega\left(z_1, \ldots, z_n\right)= a\left(\left|z_1\right|, \ldots,\left|z_n\right|\right).
$$
For a measurable function $f: \Omega \rightarrow \mathbb{C}$  and $p \in [1, \infty)$ we define 
$$
\|f\|_{p, \omega}:=\left(\int_{\Omega}|f|^p \omega d V\right)^{1/p} ,
$$
where, $d V$ denotes the standard Lebesgue volume measure on $\Omega$. If $f$ is continuous on $\Omega$ and $r= (r_1, r_2, \hdots, r_n) \in |\Omega|$, then we define $f_r$ to be the function on the unit torus $\mathbb{T}^n=\left\{(z_1, z_2, \hdots, z_n: \left|z_j\right|=1\right.$, for $\left.j=1, \ldots, n\right\} \subset \mathbb{C}^n$ given by $f_r\left(e^{i \theta_1}, \ldots, e^{i \theta_n}\right):=$ $f\left(r_1 e^{i \theta_1}, \ldots, r_n e^{i \theta_n}\right)$. 
By calculating the integral in polar coordinate and then using  Fubini's theorem we get 
$$
\|f\|_{p, \omega}^p=\int_{|\Omega|}\left\|f_r\right\|_{L^p\left(\mathbb{T}^n\right)}^p r_1 r_2 \ldots r_n \omega(r) d r.
$$
$L^p(\Omega, \omega)$ denotes the collection of measurable functions $f$ for which $\|f\|_{p, \omega}<\infty$, which forms a Banach space with respect to the norm $\|.\|_{p, \omega}$; for $p=2$, it is a Hilbert space.
Let $A^p(\Omega, \omega)=L^p(\Omega, \omega) \cap \mathcal{O}(\Omega)$, where $\mathcal{O}(\Omega)$ denotes the set of all holomorphic functions on $\Omega$.

Following the convention in \cite{chakrabarti2024projections}, a weight $\omega: \Omega \rightarrow[0, \infty]$ is said to be \textit{admissible} if for each compact set $K \subset \Omega$, there is a constant $C_K>0$ such that for each $f \in A^p(\Omega, \omega)$ we have
$$
\sup _K|f| \leq C_K\|f\|_{p, \omega}.
$$

Positive continuous functions on $\Omega$ are examples of admissible weights on $\Omega$.

If $\omega$ is a general admissible weight on $\Omega$, then $A^p(\Omega, \omega)$ forms a closed subspace of $L^p(\Omega, \omega)$, and hence a Banach space( for $p=2$, a Hilbert space). It is called a \textit{weighted $L^p$-Bergman space} on $\Omega$ with weight $\omega$. 

For $\alpha \in \mathbb{Z}^n$, $e_\alpha$ will denote the Laurent monomial function of exponent $\alpha$ defined by
$$
e_\alpha(z)=z_1^{\alpha_1} \ldots z_n^{\alpha_n} \text {, }z=(z_1, z_2, \hdots, z_n) \in \mathbb{C}^n.
$$
Following \cite{edholm2017bergman}, the set of $L^p$-allowable multi-indices for $\Omega$ with respect to $\omega$, denoted by $\Lambda(\Omega, \omega, p)$, is defined by 
$$
\Lambda(\Omega, \omega, p)=\left\{\alpha \in \mathbb{Z}^n: e_\alpha \in L^p(\Omega, \omega)\right\} .
$$

If $f \in \mathcal{O}({\Omega})$, then $f$ has a unique Laurent series expansion

$$ f(z) = \sum_{\alpha \in \mathbb{Z}^n} c_{\alpha}(f) z^{\alpha} ,   z\in \Omega$$

The series converges uniformly on compact subset of $\Omega$. In fact, it converges locally normally on $\Omega$, that is, for a compact subset $K$ the series $\sum \|c_{\alpha}(f)e_{\alpha}\|_K$ is finite, where $\|.\|_K = \text{sup}_K |.|$ (See \cite{fritzsche2002holomorphic} for a proof). The map
$ c_\alpha : \mathcal{O}(\Omega) \rightarrow \mathbb{C}$ is called the \textit{$\alpha$-th coordinate functional} of the domain $\Omega$.\\

We recall the following definition.
\begin{definition} \cite[Definition 2.1]{chakrabarti2024projections} Let $A$ be a Banach space, $n$ a positive integer, $I \subseteq \mathbb{Z}^n$ a set of multi-indices. A collection $\{f_\alpha: \alpha \in I\}$ is said to form a Banach-space basis of $A$ if for for each $ g\in A$ there are unique complex numbers $\{c_\alpha\}$ such that 

$$ g = \displaystyle \lim_{N\to \infty} \sum_{\substack{\|\alpha\|_\infty \leq N \\ \alpha \in I}} c_\alpha f_\alpha.$$
where the sequence of partial sums converges in the norm topology of $A$. 

\end{definition}

Now we state the following theorem.

\begin{theorem}\cite[Theorem 2.12]{chakrabarti2024projections}
The collection of all  Laurent monomials \linebreak
$\left\{e_\alpha: \alpha \in \Lambda(\Omega, \omega, p)\right\}$ forms a Banach-space basis of $A^p(\Omega, \omega)$. The functionals dual to this basis are the coordinate functionals $\left\{c_\alpha: \alpha \in \Lambda(\Omega, \omega, p)\right\}$, and the norm of $c_\alpha: A^p(\Omega, \omega) \rightarrow \mathbb{C}$ is given by
$$
\left\|c_\alpha\right\|_{A^p(\Omega, \omega)^{\prime}}=\frac{1}{\left\|e_\alpha\right\|_{p, \omega}} .
$$

Thus, if $f \in A^p(\Omega, \omega)$, the Laurent series of $f$ written as $\sum_{\alpha \in \mathbb{Z}^n} c_\alpha(f) e_\alpha$ consists only of terms corresponding to monomials $e_\alpha \in A^p(\Omega, \omega)$, i.e., if $\alpha \notin \Lambda(\Omega, \omega, p)$, then $c_\alpha(f)=0$.

\end{theorem}

Also, elementary computation shows that, $\left\{e_\alpha: \alpha \in \Lambda(\Omega, \omega, 2)\right\}$ forms an orthogonal basis of $A^2(\Omega, \omega)$.

\subsection{Representation of radial operators on $A^2(\Omega, \omega)$}
\hfill\\

\hspace{0.5mm} In this subsection, we characterize radial opearators on  the Hilbert space $A^2(\Omega, \omega)$(where $\omega$ is an admissible and multi-radial weight) in terms of their integral representations. 

We define the operator $T : A^2(\Omega, \omega) \longrightarrow l^2(\Lambda(\Omega, \omega, 2))$ given by 
\begin{align}
    T(f) = \left(\left\|c_\alpha\right\|_{A^2(\Omega, \omega)^{\prime}}\int_{\Omega} f(z) \bar{z}^\alpha \omega d V(z)\right)_{\alpha \in \Lambda(\Omega, \omega, 2)}
\end{align}

The operator $T$ is unitary and its adjoint $T^* : l^2(\Lambda(\Omega, \omega, 2)) \longrightarrow A^2(\Omega, \omega)$ is given by 
\begin{align}
    T^*\left(\left(c_\alpha\right)_{\alpha \in \Lambda(\Omega, \omega, 2)}\right)=\sum_{\alpha \in \Lambda(\Omega, \omega, 2)}{\|c_\alpha\|}_{A^2(\Omega, \omega)^{\prime}}  c_\alpha z^\alpha
\end{align}

For each $\xi \in \mathbb{T}^n$, the operator $V_\xi$ is a unitary operator on $A^2(\Omega, \omega).$

\begin{lemma} \label{sec 3, lem 1}
    Let $\eta = (\eta_1, \eta_2, \hdots, \eta_n) \in \mathbb{R}^n$. If $\xi= (e^{-i\eta_1}, e^{-i\eta_2}, \hdots, e^{-i\eta_n})$, then $\tau_\eta$ and $V_\xi$ are unitarily equivalent. More specifically, $T^*F\tau_\eta = V_\xi T^*F.$ 
\end{lemma}
\begin{proof}

     It is enough to show that
$$
T^* F \tau_\eta F^{-1} T e_\beta=V_{\xi} e_\beta
$$
for all $\beta \in \Lambda(\Omega, \omega, 2)$.
Let $\beta \in \Lambda(\Omega, \omega, 2)$ be fixed. Then for $\alpha \in \Lambda(\Omega, \omega, 2)$,
$$
\begin{aligned}
T e_\beta(\alpha) & =\left\|c_\alpha\right\|_{A^2(\Omega, \omega)^{\prime}} \int_{\Omega} z^\beta \overline{z^\alpha} \omega d V(z) \\
& =\left\{\begin{array}{cl}
\left\|e_\beta\right\|_{2,\omega} & \text { if } \alpha=\beta \\
0 & \text { if } \alpha \neq \beta .
\end{array}\right.
\end{aligned}
$$

Now,
$$
\begin{aligned}
T^* F \tau_\eta\left(F^{-1} T\right) e_\beta(z) 
& =T^* F \tau_\eta\left\|e_\beta\right\|_{2,\omega} e^{i (.) \cdot \beta} \\
& =\left\|e_\beta\right\|_{2,\omega}\left(T^* F\right) e^{i((.)-\eta) \cdot \beta} \\
& =\left\|c_\beta\right\|_{A^2(\Omega, \omega)^{\prime}}\left\|e_\beta\right\|_{2,\omega} e^{-i \eta \cdot \beta} z^\beta \\
& =e^{-i \eta \cdot \beta} z^\beta \\
& =V_{\xi} e_\beta(z)
\end{aligned}
$$

Hence the proof.
\end{proof}
From now onwards, $\Lambda(\Omega, \omega, 2)$ will be denoted by $\Lambda$. 

\begin{lemma}

    If $S$ commutes with $\tau_\eta$ for all $\eta \in \mathbb{R}^n$ then $\exists\left(a_\alpha\right)_{\alpha \in \Lambda} \in l^{\infty}(\Lambda)$ such that 
\begin{align}
   (S f)(\theta)=\sum_{\alpha \in \Lambda} a_\alpha \hat{f}(\alpha) e^{i \alpha \cdot \theta}, 
\end{align} for all $f \in C^{\infty}\left(\mathbb{T}^n\right)$ $\cap L_\Lambda^2\left(\mathbb{T}^n\right)$.

\begin{proof}

We note that $L^2\left(\mathbb{T}^n\right)=L_\Lambda^2\left(\mathbb{T}^n\right) \oplus L_\Lambda^2\left(\mathbb{T}^n\right)^{\perp}$. We extend the operator $S$ to $\bar{S}$ on $L^2\left(\mathbb{T}^n\right)$ by defining $\left.\bar{S}\right|_{L_{\Lambda}^2\left(\mathbb{T}^n\right)}=S$ and $\left.\bar{S}\right|_{L_{\Lambda}^2\left(\mathbb{T}^n\right)^{\perp}}=0$. It is a routine computation to check that $\bar{S}$ is bounded on $L^2\left(\mathbb{T}^n\right)$.
Note that $\tau_\eta f \in L^2_\Lambda\left(\mathbb{T}^n\right)$ for all $f \in L^2_\Lambda\left(\mathbb{T}^n\right)$ and $\eta \in \mathbb{R}^n$.
Therefore, $\bar{S}\left(\tau_\eta f\right)=S\left(\tau_\eta f\right)=\tau_\eta(S f)=\tau_\eta(\bar{S} f)$, for all $f \in L_{\Lambda}^2\left(\mathbb{T}^n\right)$.
Similarly, $\bar{S} \tau_\eta f=\tau_\eta \bar{S} f=0$, for all $f \in L_\Lambda^2\left(\mathbb{T}^n\right)^{\perp}$.
Let $f \in C^{\infty}\left(\mathbb{T}^n\right) \cap L^2_\Lambda\left(\mathbb{T}^n\right) $. Then using Theorem \ref{thm 2.12} and the fact that $S f \in L_{\Lambda}^2\left(\mathbb{T}^n\right)$, we have that,
$$ Sf(\theta) = \bar{S}f(\theta)= \sum_{\alpha \in \mathbb{Z}^n} a_\alpha \hat{f}(\alpha) e^{i \alpha \cdot \theta}=\sum_{\alpha \in \Lambda} a_\alpha \hat{f}(\alpha) e^{i \alpha \cdot \theta}.$$
Moreover, it is evident that, $\|S\|_{{{L_\Lambda^2}\left(\mathbb{T}^n\right)} \rightarrow L_\Lambda^2\left(\mathbb{T}^n\right)}=\|( a_\alpha) \|_{{l_ \infty}(\Lambda)}.$
    
\end{proof}
\end{lemma}
\begin{proposition}

    If $R\in B(A^2(\Omega, \omega))$ is a radial operator, then $R$ is unitarily equivalent to a multiplication operator on $\ell^2(\Lambda)$.
\end{proposition}

\begin{proof}
   First, we observe that $F^{-1} T R T^* F$ commutes with $\tau_\eta$ for all $\eta \in \mathbb{R}^n$ :
$$
\begin{aligned}
\left(F^{-1} T R T^* F\right) \tau_\eta 
= & \left(F^{-1} T R T^* F\right)\left(F^{-1} T V_{\xi} T^* F\right) \quad \text { [using Lemma \ref{sec 3, lem 1}] } \\
= & F^{-1} T R V_{\xi} T^* F \\
= & F^{-1} T V_{\xi} R T^* F \quad \text { [since } R \text { is radial] } \\
= & \left(F^{-1} T V_{\xi} T^* F\right)\left(F^{-1} T R T^* F\right) \\
= & \tau_\eta\left(F^{-1} T R T^* F\right) .
\end{aligned}
$$

Since $\eta \in \mathbb{R}^n$ is arbitrary, therefore the observation follows.

Let $S:=F^{-1} T R T^* F$. Then by the previous lemma, $\exists\left(a_\alpha\right)_{\alpha \in \Lambda} \in l^{\infty}(\Lambda)$ such that 
\begin{align}
   (S f)(\theta)=\sum_{\alpha \in \Lambda} a_\alpha \hat{f}(\alpha) e^{i \alpha \cdot \theta}, 
\end{align} for all $f \in C^{\infty}\left(\mathbb{T}^n\right)$ $\cap L_\Lambda^2\left(\mathbb{T}^n\right)$.

Let $a := (a_\alpha)_{\alpha \in \Lambda} \in l_\infty(\Lambda)$, and $M_a$ denote the multiplication operator on $l^2(\Lambda)$ with the symbol $a$. Then by the above observations, we have that, $FS = M_a F$, that is, $F(F^{-1} T R T^* F)=M_a F$, which implies that, 
\begin{align}\label{eq 3.4}
    R=T^*M_aT.
\end{align}

Hence the proof.

\end{proof}

\begin{definition}
    Let $\Omega$ be a Reinhardt domain in $\mathbb{C}^n$ and $\omega$ be an admissible and multi-radial weight on $\Omega$. $\Omega$ is said to be \textit{feasible} with respect to $\omega$ if 
        \begin{enumerate}

        \item $\sum_{\alpha \in \Lambda}\|c_\alpha\|_{A^2(\Omega, \omega)^{'}}^2 z^\alpha$ defines an analytic function on $\Omega$. Denote it by $g$.

        \item $g$ has an analytic extension to $\widetilde{\Omega} := \{z \bar{w} : z, w \in \Omega\}.$ 
        \end{enumerate}
\end{definition}

The analytic extension of $g$ to $\widetilde{\Omega}$ will also be denoted by $g$.
\begin{remark}\label{sec3, rem 1}
    
Let $\Omega \subseteq \mathbb{C}^n$ be a feasible domain.
For each $r \in|\Omega|$ and $\alpha \in \Lambda$, let $$S_r(\alpha):=\left\|c_\alpha\right\|_{A^2(\Omega, \omega)^{\prime}}^2 r^\alpha$$ 
Since $\Omega$ is feasible, so 
$({S_r}(\alpha))_{\alpha \in \Lambda} \in S^{\prime}(\Lambda)$ for all $r \in|\Omega|$.
Hence the series $\sum_{\alpha \in \Lambda} S_r(\alpha) e^{i \alpha \cdot \theta}$ converges in the weak-* topology of $D^{\prime}\left(\mathbb{T}^n\right)$, to $K_r$, say, for all $r \in|\Omega|$.
Therefore $\widehat{K}_r(\alpha)=S_r(\alpha)$ for all $r \in|\Omega|$.
\end{remark}

\begin{definition}
    Let $g$ be a function on a feasible domain $\Omega$ and $r= (r_1, r_2, \hdots, r_n) \in |\Omega|$, and $g_r$  be the class of  functions on the unit torus $\mathbb{T}^n$, given by $g_r\left(e^{i \theta_1}, \ldots, e^{i \theta_n}\right):=$ $g\left(r_1 e^{i \theta_1}, \ldots, r_n e^{i \theta_n}\right)$. We say that $g$ is \textit{induced by} $u \in D^{\prime}(\mathbb{T}^n)$ if $g_r$ is the sum of the Fourier series of $K_r * u$, for each $r \in |\Omega|.$ 
\end{definition}

\begin{lemma} \label{induce}
    Let $u \in D^{\prime}(\mathbb{T}^n)$ be such that the Fourier transform of $u$, i.e., $\Hat{u} \in l^\infty(\Lambda)$. If $g$ is induced by $u$ then $g$
    is analytic on $\Omega$.
\end{lemma}

\begin{proof}
    Since $g_r$ is the sum of the Fourier series of the distribution $K_r * u$, therefore we have that,
    \begin{align*} 
    g(z) = g_r\left(e^{i \theta_1}, \ldots, e^{i \theta_n}\right) & =\sum_{\alpha \in \Lambda} \widehat{(K_r * u)}(\alpha) e^{i \alpha \cdot \theta} \\
    & = \sum_{\alpha \in \Lambda} \widehat{K_r}(\alpha) \hat{u}(\alpha) e^{i \alpha \cdot \theta}\\
    & = \sum_{\alpha \in \Lambda} S_r(\alpha) \hat{u}(\alpha) e^{i \alpha \cdot \theta}
    \end{align*},
    the series is pointwise absolutely convergent, hence the result.
\end{proof}

Now we state and prove our main result of this subsection:

\begin{theorem}\label{Main theorem}
   Let $\Omega$ be a feasible domain with respect to an admissible multi-radial weight $\omega$. Then $R: A^2(\Omega, \omega) \rightarrow A^2(\Omega, \omega)$ is a bounded radial operator if and only if there exists an analytic function $g$ induced by $u$ $\in D^{\prime}(\mathbb{T}^ n)$ with $\hat{u} \in l^{\infty}(\Lambda)$ such that
    $${Rf}(z)=\int_{\Omega} f(w) g(z \bar{w}) \omega d V(w), \mbox{~for~ all~} f \in  A^2(\Omega, \omega).$$
\end{theorem}

\begin{proof}
   To prove the only if part of the statement, let us assume that $R$ is a bounded radial operator on $A^2(\Omega, \omega)$. Then the equation \ref{eq 3.4} implies that $R=T^* M_a T$, where $M_a$ denotes the multiplication operator with the symbol $a \in l^{\infty}(\Lambda)$. Since $l^{\infty}(\Lambda) \subseteq S^{\prime}(\Lambda)$, therefore there exists a unique $u \in D^{\prime}\left(\mathbb{T}^n\right)$ such that $\hat{u}=a$.

Now for $f \in A^2(\Omega, \omega)$, we have the following:
$$
\begin{aligned}
& R f(z) \\
= & \left(T^* M_a T\right)(f)(z) \\
= & \sum_{\alpha \in \Lambda}\|c_\alpha\|_{A^2(\Omega, \omega)^{\prime}}(M_a T f)(\alpha) z^\alpha . \\
= & \sum_{\alpha \in \Lambda}\|c_\alpha\|_{A^2(\Omega, \omega)^{\prime}} a(\alpha) T f(\alpha) z^\alpha \\
= & \sum_{\alpha \in \Lambda}\|c_\alpha\|_{A^2(\Omega, \omega)^{\prime}} a(\alpha)\left\|c_\alpha\right\|_{A^2(\Omega, \omega)} \int_{\Omega} f(w) \bar{w}^\alpha \omega d V(w) z^\alpha \\
= & \int_{\Omega} f(w) \sum_{\alpha \in \Lambda}\left\|c_\alpha\right\|_{A^2(\Omega, \omega)^{\prime}}^2 a(\alpha) \bar{w}^\alpha z^\alpha \omega d V(w)
\end{aligned}
$$

The interchange of the sum and integral in the last equality is justified by the use of the dominated convergence theorem.

Let $g(z):=\sum_{\alpha \in \Lambda}\|c_\alpha\|_{A^2(\Omega, \omega)^{'}}^2 a(\alpha) z^\alpha, z \in \Omega$. Then $g$ is an analytic function on $\Omega$. Since $\hat{u}=a$, therefore from the calculation in the proof of Lemma \ref{induce} , it follows that
$$
\begin{aligned}
& g(z)=g\left(r_1 e^{i \theta_1}, \ldots, r_n e^{i \theta_n}\right)=g_r\left(e^{i \theta_1}, \ldots, e^{i \theta n}\right)=\sum_{\alpha \in \Lambda} \widehat{(K_r * u)}(\alpha) e^{i \alpha \cdot \theta} \text {, } \\
\end{aligned}
$$
where $r \in |\Omega|$ and $\left(e^{i \theta_1}, \ldots, e^{i \theta_n}\right) \in \mathbb{T}^n$ .
Therefore, it follows that $u \in D^{\prime}\left(\mathbb{T}^n\right)$ induces the analytic function $g$ on $\Omega$ and since $\omega$ is feasible so we have,
$$
R f(z)=\int_{\Omega} f(w) g(z \bar{w}) \omega d V(w) \text { for all } f \in A^2(\Omega, \omega) \text {. }
$$ 

In order to prove that the stated condition in the statement is sufficient, \linebreak let us assume that $R f(z)=\int_{\Omega} f(w) g(z \bar{w}) d V(w)$, where $g \in \mathcal{O}(\Omega)$ be such that $g$ is induced by $u \in D^{\prime}\left(\mathbb{T}^n\right)$ with $\hat{u} \in l^{\infty}(\Lambda)$. From the calculation of the only if part, it follows that, $R=T^* M_{\hat{u}} T$. Since $M_{\hat{u}}$ is a bounded linear operator, therefore so is $R$.
Now we prove that $R$ is a radial operator on $A^2(\Omega, \omega)$. Let $\lambda=\left(\lambda_1, \lambda_2, \ldots, \lambda_n\right) \in \mathbb{T}^n$. Then
$$
\begin{aligned}
\left(R V_\lambda\right)(f)(z) & =\int_{\Omega}\left(V_\lambda f\right)(w) g(z \bar{w}) \omega d V(w) \\
& =\int_{\Omega} f(\lambda w) g(z \bar{w}) \omega d V(w)
\end{aligned}
$$

Applying the change of variable $u=\lambda w$, we get that the complex Jacobian of the change of variable is $\lambda_1 \lambda_2 \cdots \lambda_n$. Hence the real Jacobian, being the square of the absolute value of the complex Jacobian, is 1. Since $\omega$ is multi-radial, therefore we have,
$$
\begin{aligned}
& \left(R V_\lambda\right)(f)(z)=\int_{\Omega} f(u) g(\lambda z \bar{u}) \omega d V(u) .
\end{aligned}
$$

On the other hand, 
$$(V_\lambda R)(f)(z)=(R f)(\lambda z)=\int_{\Omega} f(w) g(\lambda z \bar{w}) \omega d V(w).$$
Since $\lambda \in \mathbb{T}^n$ was chosen arbitrarily, therefore we have that, $R V_\lambda=V_\lambda R$ for all $\lambda \in \mathbb{T}^n$. Hence $R$ is a radial operator.

\end{proof}


	\section{Example of a von Neumann algebra in $\mathcal{O}(\Omega)$} \label{SEC: p=4}
 Let us first explore some operator theoretic properties of bounded radial operators on $A^2(\Omega, \omega)$. Let $R$ be a radial operator on $A^2(\Omega, \omega)$. Then by utilising the proof of Theorem \ref{Main theorem}, it follows that, there exists a periodic distribution $u \in D^{\prime}\left(\mathbb{T}^n\right)$ with $\hat{u} \in l^{\infty}(\Lambda)$, such that $R$ is unitarily equivalent to the multiplication operator $M_{\hat{u}}$. Hence the statements $(1)-(5)$ of the following theorem hold by using the operator theoretic properties of multiplication operator on $l^2(\Lambda)$ with the symbol $\hat{u}$ listed in Theorem \ref{multiplication}, whereas the statement $(6)$ follows from \cite[Theorem 2.5]{shapiro2004notes}.

\begin{theorem}
(1) $R$ is normal.

(2) $\sigma(R)=\sigma_{app}(R)=\overline{\operatorname{range}(\hat{u})}$.

(3) $\lambda \in \sigma_p(R)$ if and only if $\hat{u}(\alpha)=\lambda$ for some $\alpha \in \Lambda$.

(4) $R$ is compact if and only if $(\hat{u}(\alpha))_{\alpha \in \Lambda}$ vanishes at infinity.

(5) $R$ is of finite rank if and only if $\hat{u}(\alpha) = 0$ for all but finitely many $\alpha \in \Lambda$.

(6) The numerical range of $R$ is the convex hull of $\sigma_p(R)$.

\end{theorem}
Also the following results hold:
\begin{theorem}
(1) The set of all radial operators on $A^2(\Omega, \omega)$ forms a maximal abelian $C^*$ - subalgebra of $B\left(A^2(\Omega, \omega)\right)$.\\
(2) A closed subspace of $A^2(\Omega, w)$ is a common reducing subspace for all radial operators on $A^2(\Omega, \omega)$ if and only if it is of the form $R\left(A^2(\Omega, \omega)\right)$, where $R=T^* M_a T$ with $a=\chi_E$ for some subset $E \subseteq \Lambda$.
\end{theorem}

\begin{proof}
    The proofs are similar to that of \cite[Theorem 4.1]{mohan2024integral}.
\end{proof}

Now we give an example of a von Newman algebra of analytic functions on an admissible Reinhardt domain $\Omega$.

A radial operator on $A^2(\Omega, \omega)$ of the form 
$$f \mapsto \int_{\Omega} f(w) g(z \bar{w}) \omega d V(w),$$ for some $g \in \mathcal{O}(\Omega)$, will be denoted by $R_g.$
For $g \in \mathcal{O}(\Omega)$, define $$g^*(z):=\overline{g(\bar{z})}  \mbox{ for all } z \in \Omega,$$ then from the Laurent series expansion of $g^*$, it follows that $g^*$ is also analytic on $\Omega$.

Let $$S:=\{g \in \mathcal{O}(\Omega): R_g \text{ is radial and bounded} \}, \text{ and }
S_R:=\left\{R_g: {R_g} \in B\left(A^2(\Omega, \omega)\right)\right\}.$$

\begin{lemma}
    
 Let $R_g \in S_R$. Then $\left(R_g\right)^*=R_{g^*}$.
 \end{lemma}
Let $f, h \in A^2(\Omega, w)$. Then
$$
\begin{aligned}
\langle R_g f, h\rangle_{A^2(\Omega, \omega)} & =\int_{\Omega}(R_g f)(w) \overline{h(w)} \omega d V(w) \\
& =\int_{\Omega} \int_{\Omega} f(u) g(w \bar{u}) \omega d V(u) \overline{h(w)} \omega d V(w) \\
& =\int_{\Omega} f(u) \overline{\int_{\Omega} h(w) g^*(u \bar{w}) \omega d V(w) \omega dV(u)}
\end{aligned}
$$

Therefore, $$(R_g)^*(f)(z)=\int_{\Omega} f(w) g^*\left(u\bar{w}) \omega d V(w),  \forall f \in A^2(\Omega, \omega)\right.$$ Hence $(R_g)^*=R_{g^*}$.

\begin{lemma}\label{Lemma 4.4}
    Let $g_1, g_2 \in S$. Then $R_{g_1}=R_{g_2}$ if and only if $g_1=g_2$.
\end{lemma}
\begin{proof}
If $R_{g_1}=R_{g_2}$, then $R_{g_1-g_2}=0$. If $K_u$ denotes the reproducing kernel at $u \in \Omega$ then
$$
\begin{aligned}
\left(R_{g_1-g_2} K_u\right)(z) & =\int_{\Omega} K_ u(w)\left(g_1-g_2\right)(z \bar{u}) \omega d V(w) \\
& =\int_{\Omega} K_u(w) \overline{\left(g_1-g_2\right)^*(\bar{z} w)} \omega d V(w) \\
& =\overline{\int_{\Omega} {\left(g_1-g_2\right)^*(\bar{z} w) \overline{K_ u(w)}} \omega d V(w)} \\
& =\overline{\left(g_1-g_2\right)^*(\bar{z} u)}, \hspace{1 mm} z \in \Omega .
\end{aligned}
$$

Therefore $R_{g_1-g_2}=0$, implies that $\left(g_1-g_2\right)^*(\bar{z} u)=0$ for all $z, u \in \Omega$. Hence by the identity theorem, we have  $\left(g_1-g_2\right)^*=0$, which implies $g_1-g_2=0$.\\

Conversely, if $g_1=g_2$, then it is obvious that $R_{g_1}=R_{g_2}$.
\end{proof}



Let $g_1, g_2 \in S$ be induced by $u_1, u_2 \in D^{\prime}\left(\mathbb{T}^n\right)$ with $\hat{u}_1, \hat{u}_2 \in l^{\infty}(\Lambda)$.
We define $g_1+g_2, g_1 \times g_2$ as the functions induced by $u_1+u_2, u_1 * u_2$ respectively. Since $u_1+u_2$, $u_1 * u_2 \in D^{\prime}\left(\mathbb{T}^n\right)$ with $\widehat{\left(u_1+u_2\right)}=\hat{u}_1+\hat{u}_2 \in l^{\infty}(\Lambda)$ and $\widehat{u_1 * u_2}=\hat{u}_1 \cdot \hat{u}_2 \in l^{\infty}(\Lambda)$, therefore $g_1+g_2, g_1 \times g_2 \in S$. Since the operation of convolution is distributive over addition in $D^{\prime}\left(\mathbb{T}^n\right)$, therefore $\times$ is distributive over $+$ in $S$. Thus $(S,+, ., \times)$ forms an algebra.

Let $\|{g}\|_{S}:=\|{R_g}\|$.
Using lemma \ref{Lemma 4.4} we get that, $\|.\|_S$ is a well defined norm on $S$, and $\left\|g_1 g_2\right\|_S=\left\|R_{g_1} R_{g_2}\right\| \leq\left\|R_{g_1}\right\|\left\|R_{g_2}\right\|=\left\|g_1\right\|_S\left\|g_2\right\|_S$.
Let $I: S \rightarrow l^{\infty}(\Lambda)$ be the map given by
$I(g)=\hat{u}$, where $u \in D^{\prime}\left(\mathbb{T}^n\right)$ is such that $u$ induces $g$.

I is bijective by Theorem \ref{Main theorem}. Also,
$$
\begin{aligned}
& I\left(g_1+g_2\right)=\widehat{(u_1+u_2)}=\hat{u}_1+\hat{u}_2=I\left(g_1\right)+I\left(g_2\right) \\
& I\left(g_1 \times g_2\right)=\widehat{u_1 * u_2}=\hat{u}_1 * \hat{u}_2=I\left(g_1\right) I\left(g_2\right) \\
& I\left(\alpha g_1\right)=\alpha \hat{u}_1=\alpha I\left(g_1\right),
\end{aligned}
$$

Now, $I\left(g^*\right)=\hat{\widetilde{u}}=\overline{\hat{u}}=\overline{I(g)}$ for all $g, \in S$, and
$$
\begin{aligned}
\|I(g)\|_{l^{\infty}(\Lambda)} & =\|\hat{u}\|_{l_\infty(\Lambda)} \\
& =\|M_{\hat{u}}\| \\
& =\left\|T^* M_{\hat{u}} T\right\| \\
& =\|R_g\| \\
& =\|g\|_S
\end{aligned}
$$

Let $I: S \rightarrow l^{\infty}(\Lambda)$ be the map given by
$I(g)=\hat{u}$, where $u \in D^{\prime}\left(\mathbb{T}^n\right)$ is such that $u$ induces $g$.

Then we have that $I$ is a $*$-isometric isomorphism.

By the above observations, we have the following result,

\begin{theorem}
$(S, +, ., \times, *, \|.\|_S)$ is a von Neumann Algebra of analytic functions on $\Omega$.

    \end{theorem}

If $\Omega = \mathbb{D}$ and $\omega =1$, then we denote the corresponding von Neumann algebra $S$ on $\mathbb{D}$ as $S^{A^2}(\mathbb{D})$. The notation will be used in the next section.

\begin{example}
        (1) Let $\omega$ be a non-zero constant weight on $\Omega$. Recall that the Bergman kernel $K_{\Omega} : \Omega \times \Omega \rightarrow \mathbb{C}$ is given by 
        $$
    K_{\Omega}(z, w) = \sum_{\alpha \in \Lambda}\left\|c_\alpha\right\|_{A^2(\Omega, \omega)^{\prime}}^2 \bar{w}^\alpha z^\alpha
    $$

It is known that the series converges uniformly on compact subsets of $\Omega \times \Omega.$ \linebreak Thus, $\sum_{\alpha \in \Lambda}\|c_\alpha\|_{A^2(\Omega, \omega)^{'}}^2 z^\alpha$ defines an analytic function on $\Omega$ and has an analytic extension to $\widetilde{\Omega} := \{z \bar{w} : z, w \in \Omega\}.$
Therefore, all Reinhardt domains are feasible with respect to a constant weight, in particular, the open unit ball, the open unit polydisc, ploy-annulus and the so-called Thullen domains are examples of feasible domains with respect to constant weight.\\
\hfill

(2) Let $\omega$ be a general admissible and multi-radial weight on a pseudoconvex Reinhardt domain $\Omega$. By \cite[Theorem 3.3]{chakrabarti2024projections} the series $\sum_{\alpha \in \Lambda}\left\|c_\alpha\right\|_{A^2(\Omega, \omega)^{\prime}}^2 \bar{w}^\alpha z^\alpha$ converges locally normally on $\Omega \times \Omega$. Thus, $\Omega$ is feasible with respect to $\omega$. In particular, the Hartogs Triangle is a feasible domain with respect to any admissible and multi-radial weight on it.
\end{example}

Thus, the integral representations of radial operators hold true on the above mentioned domains. Also, in particular, the set $S:=\{g : R_g \text{ is radial and bounded} \}$ forms a von Neumann algebra on these domains.
 
 \section{Remarks on the Hardy and the Dirichlet Spaces} \label{SEC: p=5}
 
 Let $\mathscr{D}\left(\mathbb{D}\right)$ denotes the Dirichlet space of the unit disk $\mathbb{D}$ in $\mathbb{C}$. The complete norm on this space is given by
 $$\|f\|_{\mathscr{D}(\mathbb{D})}^2=|f(0)|^2+\int_\mathbb{D}\left|f'(z)\right|^2 d A(z),$$
 for all $f \in \mathscr{D}({\mathbb{D}}),$ where $dA$ denotes the normalized Lebesgue measure on $\mathbb{D}.$

 Let $\mathbb{Z}_{+}$ denotes the set of all non-negative integers. Following the similar line of construction as in the case of the weighted Bergman spaces of Reinhardt domains, one obtains,

\begin{theorem}
    
$R: \mathscr{D}\left(\mathbb{D}\right) \rightarrow \mathscr{D}\left(\mathbb{D}\right)$ is a bounded radial operator if and only if there exists a distribution $u \in D^{\prime}\left(\mathbb{T}\right)$ with $(\hat{u}(m))_{m \in \mathbb{Z}_{+}} \in l^{\infty}\left(\mathbb{Z}_{+}\right)$ which induces the analytic function $g$ on $\mathbb{D}$ such that
$$
R f(z)=\hat{u}(0) f(0)+z \int_{\mathbb{D}} f'(w) g'(z \bar{w}) d A(w) .
$$
for all $f \in \mathscr{D}(\mathbb{D}).$  
\end{theorem}

Due to the dependency of $R$ on the Dirichlet space and the analytic function $g$, let us denote the operator $R$ as $R^{\mathscr{D}}_{g}.$ Let $S^{\mathscr{D}}(\mathbb{D}) :=\{ g \in \mathcal{O}(\mathbb{D}): R^{\mathscr{D}}_{g} \text{ is bounded} \}.$ Then $S^{\mathscr{D}}$ is a von Neumann algebra on $\mathbb{D}$, where the algebraic operations and the norm on $S^{\mathscr{D}}$ are defined similiarly as in the previous section.

Let $H^2(\mathbb{D})$ denotes the Hardy-Hilbert space of $\mathbb{D}=\{z \in \mathbb{C}:|z|<1\}$. Then it is known that, 
$$\|f\|_{H^2}^2=|f(0)|^2+\int_\mathbb{D}\left|f^{\prime}(z)\right|^2 \log \frac{1}{|z|^2} d A(z),$$
for all $ f \in H^2(\mathbb{D}).$
Hence, proceeding similarly as the proof of Theorem \ref{Main theorem}, we have the following result for the Hardy-Hilbert space of the unit disk:

\begin{theorem}
    $\tilde{R}: {H}^2(\mathbb{D}) \rightarrow {H}^2(\mathbb{D})$ is a bounded radial operator if and only if there exists a distribution $u \in D^{\prime}(\mathbb{T})$ with $(\hat{u}(m))_{m \in \mathbb{Z}_{+}} \in l^{\infty}\left(\mathbb{Z}_{+}\right)$ which induces the analytic function $g$ on $\mathbb{D}$ such that
    
    $$
    \tilde{R} f(z)=\hat{u}(0) f(0)+z \int_{\mathbb{D}} f^{\prime}(w) g^{\prime}(z \bar{w}) \log \frac{1}{|z|^2} d A(w)
    $$
    for all $f \in H^2(\mathbb{D})$.
\end{theorem}

Let us denote $\tilde{R}$ as $R^{H^2}_g$, and $S^{H^2}(\mathbb{D}) := \{g \in \mathcal{O}(\mathbb{D}): R^{H^2}_g \in B(H^2(\mathbb{D}))\}.$ Then $S^{H^2}$ is also a von Neumann algebra of analytic functions on $\mathbb{D}$, where the algebraic operations and the norm on $S^{H^2}$ are similar as in the previous section.

Recall that the \textit{disc algebra} on $\mathbb{D}$, denoted by $A(\mathbb{D})$, is the collection of all continuous functions on the closure of $\mathbb{D}$ which are analytic on $\mathbb{D}$, and let $H^\infty (\mathbb{D})$ denotes the set of all bounded analytic functions on $\mathbb{D}$. Then we have the following chain of inclusion relations.

$$ A(\mathbb{D}) \subseteq H^\infty (\mathbb{D}) \subseteq H^2(\mathbb{D}) \subseteq S^{H^2}(\mathbb{D}),$$

$$H^\infty (\mathbb{D}) \subseteq A^2(\mathbb{D}) \subseteq S^{A^2}(\mathbb{D}),$$

$$ H^\infty (\mathbb{D}) \subseteq \mathscr{D}(\mathbb{D}) \subseteq S^{\mathscr{D}}(\mathbb{D}),$$

and 

$$ S^{\mathscr{D}}(\mathbb{D}) \subseteq S^{H^2}(\mathbb{D}) \subseteq S^{A^2}(\mathbb{D}).$$

Note that the each of the above inclusions is strict, and each of the algebras in the last chain of inclusions is isometrically isomorphic to $l^\infty(\mathbb{Z}_{+}).$

	\section*{Acknowledgments}
	
	 I am indebted to my supervisor Dr. Shibananda Biswas for his unwavering support, guidance and useful suggestions that improved the quality of this manuscript.


	\bibliographystyle{abbrv}
    \bibliography{Radial} 
	
\end{document}